\documentclass{article} 

\usepackage{amsmath,amsthm,amssymb,amsfonts}
\usepackage{graphicx}
\usepackage{geometry} 
\newtheorem{Theorem}{Theorem}[section]
\theoremstyle{plain}
\newtheorem{Example}[Theorem]{Example}

\newtheorem{Definition}[Theorem]{Definition}
\newtheorem{Remark}[Theorem]{Remark}
\newtheorem{Corollary}[Theorem]{Corollary}
\newtheorem{Proposition}[Theorem]{Proposition}

\numberwithin{equation}{section}

\newcommand{\E}{\mathbb{E}}

\title{Long edges in Galton-Watson trees}
\author{Sergey Bocharov\footnote{S.Bocharov: Department of Foundational Mathematics, Xian Jiaotong-Liverpool University, Ren'Ai Road 111, Suzhou 215123, China, e-mail: Sergey.Bocharov@xjtlu.edu.cn.} 
\ and Simon C. Harris\footnote{S.C.Harris: Department of Statistics, University of Auckland, 38 Princes Street, Auckland, 1001, New Zealand, e-mail: simon.harris@auckland.ac.nz} 
}
\begin{document}
\maketitle
\begin{abstract}
In this article, we will establish a number of results concerning the limiting behaviour of the longest edges in the genealogical tree generated by a continuous-time Galton-Watson (GW) process. Separately,  we consider the large time behaviour of the longest pendant edges, the longest (strictly) interior edges, and the longest of all the edges. 
These results extend the special case of long pendant edges of birth-death processes established in Bocharov, Harris, Kominek, Mooers, and  Steel \cite{BHMKS22} .

\end{abstract}

\section{Introduction}
We consider a continuous-time Galton-Watson process with branching rate $\beta>0$ and offspring distribution $(p_k)_{k \geq 0}$ initiated by a single particle.
Such process may be constructed as follows. It starts with a single particle at time $0$. This initial particle lives for a random time $T$ which is exponentially distributed with rate parameter $\beta$. That is, $\mathbb{P}(T > t) = \mathrm{e}^{- \beta t}$ and $\mathbb{E} T = {1}/{\beta}$. 
At the moment when the initial particle dies, it produces a random number $\xi$ of new particles where $\mathbb{P}(\xi = k) = p_k$ for every integer $k \geq 0$. Newly-born particles, independently of each other and of all the previous history, replicate the initial particle's behaviour, that is, they live for random times with $Exp(\beta)$ distribution, produce random numbers of new particles according to the offspring distribution $(p_k)_{k \geq 0}$ at the moment of their death, and so on.
Graphically such a process is most naturally represented by a tree in which each edge corresponds to the life span of a particle, as in Figure \ref{tree}. We use $\mathcal{T}_t$ to denote the entire process (tree) evolved up to time $t$.

\begin{figure}[htbp]\label{tree}
\begin{center}
\includegraphics[scale=0.65]{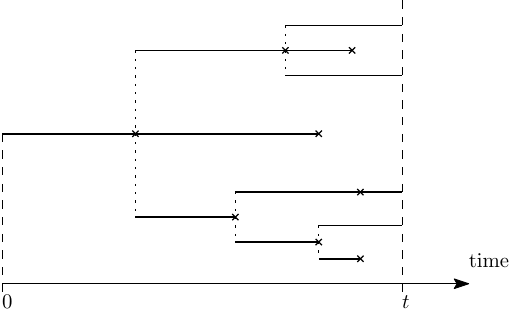}
\caption{A typical realisation of $\mathcal{T}_t$.}
\end{center}
\end{figure}

In this article we shall study the limiting behaviour of the longest edges in  $\mathcal{T}_t$ as $t \to \infty$. We classify edges in $\mathcal{T}_t$ as either pendant (edges corresponding to particles alive at time $t$) or interior (all other edges). If one thinks of $\mathcal{T}_t$ as describing some sort of evolutionary process then pendant edges would correspond to extant species and interior edges would correspond to extinct species. 

Throughout, we assume the generic offspring random variable $\xi$ has mean $m:=\E(\xi)>1$ (the Galton-Watson process is  \emph{supercritical} and has some strictly positive chance of surviving forever) 
as well as  the classical condition that $\E(\xi\log \xi)<\infty$. 
(Note, these assumptions permit cases where the offspring variance can be infinite.)
Conditional on survival of the population, for any real number $x$, we will show that, for large times $t$, the number of \emph{pendant edges} longer than $(1-1/m)t +x$ converges in distribution to a mixture of Poisson random variables with (random) parameter $e^{-m\beta x}M_\infty$, where the random variable $M_\infty$ is the classical additive martingale limit for the GW process. 
(Intuitively speaking, $M_\infty$ roughly measures the growth of the process during its earlier stages, alternatively, $\log M_\infty$ can be thought of as a random time delay for the process to get very large and start behaving deterministically).
As a corollary, we show the length of the $k^{th}$ longest pendant edge at large times $t$ converges in distribution when centred around $(1-1/m)t$ and we give an explicit representation for this limiting law. In particular,  the length of the longest pendant edge at time $t$ minus $(1-1/m)t$ converges in the large time $t$ limit to a mixture of Gumbel distributions. 

Similarly, we will prove analogous results for the longest \emph{interior edges} and the longest of \emph{all edges}. 
It turns out that long edges should always be centred around  size $(1-1/m)t$ and always converge to a mixture of Poisson distributions, where only the mixture parameter varies according to whether long pendant, long interior, or longest of all edges are being considered, respectively. 
The analysis and proofs for the longest interior edges, and similarly longest of all edges, turn out to be more delicate and require considerably more involved calculations than the longest pendant edges case (where there is much more independence). The \emph{LlogL} moment assumption is the expected to be best possible for our results to hold in  light of the natural appearance of $M_\infty$ in the description of the limit and the classical Kesten-Stigum theorem for GW processes. However, as this assumption permits the offspring variance to be infinite in general,  it necessitates an additional careful truncation argument to allow application of a second moment method involving use of the Paley-Zygmund inequality.

\subsection{Preliminaries}
We use $m$ to denote the mean of the offspring distribution:
\[
m := \mathbb{E} \xi = \sum_{k \geq 0} kp_k
\]
As we shall see, this is the only quantity which determines the first-order limiting behaviour of all the longest edges in the tree. We also use $v$ to denote the second moment of the offspring distribution: 
\[
v := \mathbb{E} \big[ \xi^2 \big] = \sum_{k \geq 0} k^2p_k
\]
In this article it is always assumed that $m < \infty$, whereas $v$ may be finite or infinite.

We denote the number of particles alive at time $t$ by $N_t$ and the event of the population surviving forever by 
\[
\{ \text{survival} \} := \bigcap_{t \geq 0} \big\{N_t > 0 \big\}.
\]
We let $\{$extinction$\} = \{$survival$\}^c$ be the event of the population eventually becoming extinct.

It is well-known that the process 
\begin{equation}
\label{martingale}
M_t = \mathrm{e}^{-(m-1)\beta t} N_t \quad , t \geq 0
\end{equation}
is a (positive) martingale, which then has an almost sure limit $M_\infty = \lim_{t \to \infty} M_t$. 

The following assumptions will be made in this article:
\begin{equation}
\label{supercritical}
m > 1 \tag{A1}
\end{equation}
and
\begin{equation}
\label{xlogx}
\mathbb{E} \big[ \xi \log^+ \xi \big] = \sum_{k \geq 1} k \log k \ p_k < \infty \tag{A2}
\end{equation}
Condition \eqref{supercritical} is necessary and sufficient for $\mathbb{P}($survival$) > 0$ and is known as the supercritical case. If condition \eqref{supercritical} is satisfied then the celebrated Kesten-Stigum Theorem says that condition \eqref{xlogx} is necessary and sufficient for $\mathbb{P}(M_\infty > 0 | $survival$) = 1$ or, in other words, for events $\{M_\infty > 0\}$ and $\{$survival$\}$ to agree almost surely (see e.g. \cite{O98}). This gives a fine approximation of the population growth: $N_t \sim M_\infty \mathrm{e}^{(m-1)\beta t}$ as $t \to \infty$ a.s. on survival (the words `a.s. on survival' are commonly used to mean `$\mathbb{P}( \cdot | $survival$) = 1$'). Let us remark that if condition $\eqref{xlogx}$ is not satisfied then $\mathbb{P}(M_\infty=0 ) = 1$ and establishing the precise growth rate of $N_t$ is a difficult task.
\subsection{Main results}
We let
\begin{equation}
\label{alpha}
\alpha^\ast := 1 - \frac{1}{m}.
\end{equation}
As we shall soon see, all the longest edges in $\mathcal{T}_t$ are approximately of length $\alpha^\ast t$ for $t$ large. 
\begin{Definition}
\label{NN}
For a length $l \geq 0$ and a time $t \geq 0$ we define $N_t^l$, $\hat{N}_t^l$, and $\tilde{N}_t^l$ respectively to be the numbers of pendant, interior, and all (counting both pendant and interior) edges in $\mathcal{T}_t$ of length $\geq l$. Note, in what follows we shall often take $l$ to be dependent on $t$. 
\end{Definition}
For all the results stated below we assume that conditions \eqref{supercritical} and \eqref{xlogx} hold. We let $\alpha^\ast$ be as in \eqref{alpha} and $M_\infty$ be the almost sure limit of the martingale $M_t$ in \eqref{martingale}.
\begin{Theorem}
\label{N_distribution}
Let $x \in \mathbb{R}$ be fixed, $t \geq 0$ variable and let us take
\[
l = l(t) = \alpha^\ast t + x.
\]
Then, under $\mathbb{P}(\cdot | $survival$)$,
\begin{equation}
\label{N_distribution_eq}
N_t^l \Rightarrow V_x , \quad \hat{N}_t^l \Rightarrow \hat{V}_x \ \text{ and } \tilde{N}_t^l \Rightarrow \tilde{V}_x
\end{equation}
as $t \to \infty$, where $V_x$, $\hat{V}_x$ and $\tilde{V}_x$ are mixtures of Poisson distributions with respective rates $\mathrm{e}^{-m \beta x} M_\infty$, $\frac{1}{m-1} \mathrm{e}^{-m \beta x} M_\infty$ and $\frac{m}{m-1} \mathrm{e}^{-m \beta x} M_\infty$ conditional on $M_\infty >0$. To be precise, for $k = 0$, $1$, $2$, ...,
\begin{equation}
\label{V_pendant}
\mathbb{P} \big( V_x = k \big\vert \text{survival} \big) = \mathbb{E} \Big[ \frac{1}{k!} \Big( \mathrm{e}^{-m \beta x} M_\infty \Big)^k \exp \big\{ - \mathrm{e}^{-m \beta x} M_\infty \big\}\Big\vert M_\infty>0 \Big],
\end{equation}
\begin{equation}
\label{V_interior}
\mathbb{P} \big( \hat{V}_x = k \big\vert \text{survival} \big) = \mathbb{E} \Big[ \frac{1}{k!} \Big( \frac{\mathrm{e}^{-m \beta x} M_\infty}{m - 1}  \Big)^k \exp \Big\{ - \frac{ \mathrm{e}^{-m \beta x} M_\infty}{m - 1} \Big\}\Big\vert M_\infty>0 \Big],
\end{equation}
and 
\begin{equation}
\label{V_all}
\mathbb{P} \big( \tilde{V}_x = k \big\vert \text{survival} \big) = \mathbb{E} \Big[ \frac{1}{k!} \Big( \frac{m \mathrm{e}^{-m \beta x} M_\infty}{m-1}  \Big)^k \exp \Big\{ - \frac{m \mathrm{e}^{-m \beta x} M_\infty}{m-1}  \Big\}\Big\vert M_\infty>0 \Big].
\end{equation}
\end{Theorem}
Let us comment that more can be shown to be true. For example, if we take $l_1 = \alpha^\ast t + x$ and $l_2 = \alpha^\ast t + y$ for $x<y$ then $(N_t^{l_1} , N_t^{l_2}) \Rightarrow (V_x, V_y)$, where $(V_x, V_y-V_x)$ is a mixture of pair independent Poisson distributions with rates $\mathrm{e}^{-m \beta x} M_\infty$ and $(\mathrm{e}^{-m \beta x} - \mathrm{e}^{-m \beta y})M_\infty$ conditional on $M_\infty > 0$.

From Theorem \ref{N_distribution} one can now recover the limiting distributions of all longest edges in $\mathcal{T}_t$ (shifted by $\alpha^\ast t$).
\begin{Definition}
\label{LL}
For $k = 1$, $2$, $3$, ... let $L_t^{(k)}$, $\hat{L}_t^{(k)}$ and $\tilde{L}_t^{(k)}$ be the lengths of $k$th longest edge among pendant, interior and all edges in $\mathcal{T}_t$ respectively. 
\end{Definition}
\begin{Corollary}
\label{L_distribution}
For all $k=1$, $2$, $3$, ..., under $\mathbb{P}(\cdot | $survival$)$,
\begin{equation}
\label{L_distribution_eq}
L_t^{(k)} - \alpha^\ast t \Rightarrow W^{(k)}, \quad \hat{L}_t^{(k)} - \alpha^\ast t \Rightarrow \hat{W}^{(k)} \ \text{ and }
\tilde{L}_t^{(k)} - \alpha^\ast t \Rightarrow \tilde{W}^{(k)},
\end{equation}
where 
\[
\mathbb{P}\big(W^{(k)} \leq x | \text{survival}\big) = \mathbb{E} \Big[ \sum_{j=0}^{k-1} \frac{1}{j!} \Big( \mathrm{e}^{-m \beta x} M_\infty \Big)^j \exp \big\{ \!-\! \mathrm{e}^{-m \beta x} M_\infty \big\}\Big\vert M_\infty>0 \Big],
\]
\[
\mathbb{P}\big(\hat{W}^{(k)} \leq x | \text{survival}\big) = \mathbb{E} \Big[ \sum_{j=0}^{k-1} \frac{1}{j!} \Big( \frac{\mathrm{e}^{-m \beta x} M_\infty}{m - 1}  \Big)^j \exp \Big\{ - \frac{ \mathrm{e}^{-m \beta x} M_\infty}{m - 1} \Big\} \Big\vert M_\infty>0 \Big]
\]
and
\[
\mathbb{P}\big(\tilde{W}^{(k)} \leq x | \text{survival}\big) = \mathbb{E} \Big[ \sum_{j=0}^{k-1} \frac{1}{j!} \Big( \frac{m \mathrm{e}^{-m \beta x} M_\infty}{m-1}  \Big)^j \exp \Big\{ - \frac{m \mathrm{e}^{-m \beta x} M_\infty}{m-1}  \Big\}\Big\vert M_\infty>0 \Big]\]
\end{Corollary}
In particular, the limiting distributions of the longest pendant, interior and any edge centred around $\alpha^\ast t$ are mixtures of Gumbel distributions:
\[
\mathbb{P}\big(W^{(1)} \leq x | \text{survival}\big) = \mathbb{E} \Big[ \exp \big\{ \!-\! \mathrm{e}^{-m \beta x} M_\infty \big\}\Big\vert M_\infty>0 \Big]
\]
(and similarly for $\hat{W}^{(1)}$ and $\tilde{W}^{(1)}$).
\begin{proof}[Proof of Corollary \ref{L_distribution}]
We note that 
\[
\mathbb{P}\big(L_t^{(k)} > \alpha^\ast t + x \big\vert \text{survival}\big) = \mathbb{P}\big(L_t^{(k)} \geq \alpha^\ast t + x \big\vert \text{survival}\big) = 
\mathbb{P}\big(N_t^l \geq k \big\vert \text{survival} \big), 
\]
where $l = \alpha^\ast t + x$. Hence, combining this observation with Theorem \ref{N_distribution},
\[
\mathbb{P}(L_t^{(k)} - \alpha^\ast t \leq x | \text{survival}) = \mathbb{P} \big( N_t^l < k \big\vert \text{survival} \big) \to \mathbb{P} \big( V_x  < k \big\vert \text{survival} \big)
\]
as $t \to \infty$. The proofs for $\hat{L}_t^{(k)}$ and $\tilde{L}_t^{(k)}$ are essentially the same.
\end{proof}
\begin{Remark}
Let us note that conditioning on the process surviving forever may be easily replaced with conditioning on survival up to time $t$ whenever it would make sense since $\mathbb{P}(N_t>0) \searrow \mathbb{P}($survival$)$ as $t \to \infty$.  For example, $\mathbb{P}(V_x=k | $survival$)$ in Theorem \ref{N_distribution} may be replaced with $\lim_{t \to \infty} \mathbb{P}(N_t^l = k | N_t>0)$ and likewise for interior and all edges. Similarly, $\mathbb{P}(W^{(k)} \leq x |$ survival $)$ in Corollary \ref{L_distribution} may be replaced with $\lim_{t \to \infty} \mathbb{P}(L_t^{(k)}-\alpha^\ast t \leq x | N_t > 0)$ and likewise for interior and all edges. 

It is also easy to see (after dividing \eqref{L_distribution_eq} by $t$ and using some standard results about convergence of random variables) that $\frac{L_t^{(k)}}{t} \to \alpha^\ast$ in $\mathbb{P}(\cdot | $survival$)$-probability. Or, for the same reason as above, $\lim_{t \to \infty} \mathbb{P}\big( \big\vert \frac{L_t^{(k)}}{t} - \alpha^\ast \big\vert > \epsilon \big\vert N_t > 0 \big) = 0$ for all $\epsilon > 0$ and likewise for interior and pendant edges. In fact, a stronger result is true.
\end{Remark}
\begin{Theorem}
\label{a_s}
For $k = 1$, $2$, $3$, ...,  
\begin{equation}
\label{a_s_limit}
\frac{L_t^{(k)}}{t} \to \alpha^\ast, \quad \frac{\hat{L}_t^{(k)}}{t} \to \alpha^\ast \ \text{ and } \frac{\tilde{L}_t^{(k)}}{t} \to \alpha^\ast
\end{equation}
as $t \to \infty$ a.s. on survival.
\end{Theorem}

\begin{subsubsection}{Special case: Long edges in the Birth-Death process}
We can specialise the above  results about long edges in general Galton-Watson results to the Birth-Death process. The Birth-Death process is not only an important model for many applications, but it also permits further explicit calculations and simplified expressions. 
Previously, only the behaviour of long pendant edges for Birth-Death processes had been established in \cite{BHMKS22}. 
\begin{Example}
Consider a supercritical birth-death process with birth rate $\lambda$ and death rate $\mu$ ($\lambda > \mu$). This is a special case of a Galton-Watson process with branching rate $\beta = \lambda + \mu$ and the offspring distribution $p_0 = \frac{\mu}{\lambda + \mu}$, $p_2 = \frac{\lambda}{\lambda+\mu}$. 

In this case $m = \frac{2 \lambda}{\lambda + \mu}$ and so
\[
\alpha^\ast = 1 - \frac{\lambda + \mu}{2 \lambda} = \frac{\lambda - \mu}{2 \lambda} = \frac{1}{2} \big( 1 - \frac{\mu}{\lambda} \big).
\]
Moreover, in this case $M_\infty$ has $Exp(1 - \frac{\mu}{\lambda})$ distribution conditional on survival (see e.g. \cite{H64}).

One can then easily check that $V_x$, $\hat{V}_x$, and $\tilde{V}_x$ from Theorem \ref{N_distribution} follow Geometric distribution on $\{0, 1, 2, \cdots\}$ with parameters $\frac{1}{1+ \frac{\lambda}{\lambda - \mu} \mathrm{e}^{-m \beta x}}$, $\frac{1}{1+ \frac{\lambda}{\lambda - \mu} \frac{\mathrm{e}^{-m \beta x}}{m-1}}$, and $\frac{1}{1+ \frac{\lambda}{\lambda - \mu} \frac{m \mathrm{e}^{-m \beta x}}{m-1}}$, respectively.

In addition, the cumulative distribution functions of $W^{(k)}$, $\hat{W}^{(k)}$ and $\tilde{W}^{(k)}$ from Corollary \ref{L_distribution} become:
\[
\mathbb{P}\big(W^{(k)} \leq x | \text{survival}\big) = 1 - \frac{1}{\big( 1 + \frac{\lambda - \mu}{\lambda} \mathrm{e}^{m \beta x} \big)^k},
\]
\[
\mathbb{P}\big(\hat{W}^{(k)} \leq x | \text{survival}\big) = 1 - \frac{1}{\big( 1 + \frac{\lambda - \mu}{\lambda} (m-1) \mathrm{e}^{m \beta x} \big)^k}
\]
and
\[
\mathbb{P}\big(\tilde{W}^{(k)} \leq x | \text{survival}\big) = 1 - \frac{1}{\big( 1 + \frac{\lambda - \mu}{\lambda} \frac{m-1}{m} \mathrm{e}^{m \beta x} \big)^k}.
\]
In particular, $W^{(1)}$, $\hat{W}^{(1)}$ and $\tilde{W}^{(1)}$ all follow Logistic distributions with the same scale parameter $\frac{1}{m \beta}$ but different location parameters $\frac{1}{m \beta} \log (\frac{\lambda}{\lambda - \mu})$, $\frac{1}{m \beta} \log (\frac{1}{m-1} \cdot \frac{\lambda}{\lambda - \mu})$, and $\frac{1}{m \beta} \log (\frac{m}{m-1} \cdot \frac{\lambda}{\lambda - \mu})$, respectively. 
\end{Example}
\end{subsubsection}
\begin{subsubsection}{Structure of this article} 
The rest of the article is organised as follows. A simple proof of Theorem \ref{N_distribution} for pendant edges will be presented in Section 2. The proof of Theorem \ref{N_distribution} for interior and all edges will be given in Section 3 with all the necessary heavy calculations presented in Section 4 separately. Finally, Theorem \ref{a_s} will be proved in Section 5. 
\end{subsubsection}
\section{Proof of Theorem \ref{N_distribution} for pendant edges.}
Recall that $N_t^l$ is the number of pendant edges of length $\geq l$ at time $t$. 
\begin{proof}[Proof of Theorem \ref{N_distribution} for pendant edges]
We observe that for $t>l$
\begin{equation}
\label{N_t_l}
N_t^l = \sum_{k=1}^{N_{t-l}} \mathbf{1}_{A_k},
\end{equation}
where 
\[
A_k = \big\{ k\text{th particle alive at time } t-l \text{ does not die in the time interval } [t-l, t] \big\}.
\]
Events $A_k$, $k \in N_{t-l}$ are independent under $\mathbb{P} (\cdot | \mathcal{F}_{t-l})$ with 
\[
\mathbb{P}(A_k | \mathcal{F}_{t-l}) = \mathrm{e}^{- \beta l}.
\]
We define $m_t^l := \mathbb{E}N_t^l$ and $v_t^l = \mathbb{E} \big( N_t^l \big)^2$. If we then recall the facts (see e.g. \cite{H64}) that 
\begin{equation}
\label{m_t}
\mathbb{E} N_\tau = \mathrm{e}^{\beta (m-1) \tau}
\end{equation}
and, assuming $v<\infty$, 
\begin{equation}
\label{v_t}
\mathbb{E} \big[ N_\tau(N_\tau - 1) \big] = \frac{v-m}{m-1} \Big( \mathrm{e}^{2\beta (m-1) \tau} - \mathrm{e}^{\beta (m-1) \tau} \Big)
\end{equation}
for any $\tau \geq 0$ then it follows from \eqref{N_t_l} that
\begin{equation}
\label{m_t_l}
m_t^l = \mathrm{e}^{-\beta l} \mathbb{E} N_{t-l} = \mathrm{e}^{-\beta l} \cdot \mathrm{e}^{\beta (m-1)(t-l)} = \mathrm{e}^{-\beta(ml - (m-1)t)}
\end{equation}
and, assuming $v<\infty$,  
\begin{align}
\label{v_t_l}
v_t^l = &\mathbb{E} \Big[ \sum_{k=1}^{N_{t-l}} \mathbf{1}_{A_k} + \sum_{\substack{i,j = 1\\i \neq j}}^{N_{t-l}} \mathbf{1}_{A_i} \mathbf{1}_{A_j} \Big]\\
= &\mathrm{e}^{-\beta l} \mathbb{E} N_{t-l} + \mathrm{e}^{-2 \beta l} \mathbb{E} \big[ N_{t-l} (N_{t-l} - 1) \big] \nonumber \\
= &\mathrm{e}^{-\beta l} \cdot \mathrm{e}^{\beta (m-1)(t-l)} + \frac{v-m}{m-1} \mathrm{e}^{-2 \beta l} \Big( 
\mathrm{e}^{2 \beta (m-1)(t-l)} - \mathrm{e}^{\beta (m-1)(t-l)}\Big)\nonumber\\
= &\mathrm{e}^{-\beta(ml - (m-1)t)} \Big[ 1 + \frac{v-m}{m-1} \mathrm{e}^{-\beta(ml - (m-1)t)}  - \frac{v-m}{m-1} \mathrm{e}^{- \beta l} \Big]
\end{align}
Moreover, the probability generating function of $N_t^l$ satisfies
\[
\mathbb{E} \Big[ \theta^{N_t^l} \Big] = \mathbb{E} \Big[ \theta^{\sum_{k=1}^{N_{t-l}} \mathbf{1}_{A_k}} \Big] = 
\mathbb{E} \Big[ \prod_{k=1}^{N_{t-l}} \theta^{\mathbf{1}_{A_k}} \Big] = \mathbb{E} \Big[ \Big( \theta \mathbb{P}(A_k) + 1 - \mathbb{P}(A_k) \Big)^{N_{t-l}} \Big] = \mathbb{E} \Big[ \Big( 1+(\theta-1) \mathrm{e}^{-\beta l} \Big)^{N_{t-l}} \Big],
\]
$\theta \in (-1,1)$. If we now take 
\[
l = l(t) = \alpha^\ast t + x = \big( 1 - \frac{1}{m} \big)t + x
\]
for some fixed $x$ and a variable $t$ then 
\begin{align*}
\mathbb{E} \Big[ \Big( 1 + (\theta-1) \mathrm{e}^{-\beta l} \Big)^{N_{t-l}} \Big]
= &\mathbb{E} \Big[ \Big( \big( 1 + (\theta-1) \mathrm{e}^{-\beta l}\big)^{\mathrm{e}^{\beta(m-1)(t-l)}} \Big)^{\mathrm{e}^{-\beta(m-1)(t-l)} N_{t-l}} \Big]\\
= &\mathbb{E} \Big[ \Big( \big( 1 + (\theta-1) \mathrm{e}^{-\beta l}\big)^{\mathrm{e}^{\beta(m-1)(t-l)}} \Big)^{M_{t-l}} \Big]\\
= &\mathbb{E} \Big[ \Big( \big( 1 + (\theta-1) \mathrm{e}^{-m \beta x} \mathrm{e}^{-\beta(m-1)(t-l)}\big)^{\mathrm{e}^{\beta(m-1)(t-l)}} \Big)^{M_{t-l}} \Big]\\
\to &\mathbb{E} \Big[ \exp \big\{ (\theta-1) \mathrm{e}^{-m \beta x} M_\infty \big\} \Big]
\end{align*}
as $t \to \infty$ using the facts that $\mathrm{e}^{\beta (m-1)(t-l)} \to \infty$ as $t \to \infty$, $(1 + \frac{c}{\tau})^\tau \to \mathrm{e}^c$ as $\tau \to \infty$, $M_{t-l} \to M_\infty$ as $t \to \infty$ a.s. on survival and bounded convergence theorem. It also follows that
\begin{align*}
\mathbb{E} \big[ \theta^{N_t^l} \big\vert \text{survival} \big] = &\frac{\mathbb{E} \big[ \theta^{N_t^l} \big] - \mathbb{E} \big[ \theta^{N_t^l} \mathbf{1}_{\{ extinction \}}\big]}{\mathbb{P}(survival)}\\
\to &\frac{\mathbb{E} \big[ \exp \{ (\theta-1) \mathrm{e}^{-m \beta x} M_\infty \} \big] - \mathbb{P}(extinction)}{\mathbb{P}(survival)}\\
= &\frac{\mathbb{E} \big[ \exp \{ (\theta-1) \mathrm{e}^{-m \beta x} M_\infty \} \big] - \mathbb{P}(M_\infty = 0)}{\mathbb{P}(M_\infty > 0)}\\
= &\mathbb{E} \Big[ \exp \big\{ (\theta-1) \mathrm{e}^{-m \beta x} M_\infty \big\} \Big\vert M_\infty > 0 \Big]
\end{align*}
recalling that events $\{$extinction$\}$ and $\{M_\infty > 0\}$ are almost surely equal. This establishes \eqref{N_distribution_eq} for $N_t^l$ since $\exp \big\{ (\theta-1) \mathrm{e}^{-m \beta x} M_\infty \big\}$ is the probability generating function of a Poisson random variable with parameter $\mathrm{e}^{- m \beta x} M_\infty$.
\end{proof}
\section{Proof of Theorem \ref{N_distribution} for interior and all edges.}
Recall that $\hat{N}_t^l$ and $\tilde{N}_t^l$ are the numbers of interior and all edges respectively of length $\geq l$ in the tree $\mathcal{T}_t$. \begin{Proposition}
\label{m_hat_tilde}
Let $\hat{m}_t^l = \mathbb{E} \hat{N}_t^l$ and $\tilde{m}_t^l = \mathbb{E} \tilde{N}_t^l$. Then
\begin{equation}
\label{m_hat}
\hat{m}_t^l =  \left\{
\begin{array}{rl}
0 & \text{if } t \leq l\\
\frac{1}{m-1} \mathrm{e}^{-\beta l} \big( \mathrm{e}^{\beta (m-1)(t- l)} - 1\big) & \text{if } t >  l
\end{array} \right.
\end{equation}
and
\begin{equation}
\label{m_tilde}
\tilde{m}_t^l =  \left\{
\begin{array}{rl}
0 & \text{if } t \leq l\\
\frac{1}{m-1} \mathrm{e}^{-\beta l} \big( m \mathrm{e}^{\beta (m-1)(t- l)} - 1\big) & \text{if } t >  l
\end{array} \right.
\end{equation}
\end{Proposition}
\begin{proof}
Trivially $\hat{m}_t^l  =\tilde{m}_t^l  =0$ when $t \leq l$ because edges in $\mathcal{T}_t$ cannot be longer than $t$. For the rest of the proof let $l$ be fixed and let $t>l$. If $T$ is the branching time of the initial particle (so that $\mathbb{P}(T>t) = \mathrm{e}^{- \beta t}$) and $\xi$ is the number of children produced by the initial particle then
\begin{align}
\label{N}
\hat{N}_t^l = &0 \cdot \mathbf{1}_{\{T > t\}} + \Big[ 1 + \sum_{k=1}^\xi \hat{N}_{t-T}^l(k) \Big] \mathbf{1}_{\{l <T \leq t\}} + 
\Big[ 0 + \sum_{k=1}^\xi \hat{N}_{t-T}^l(k) \Big] \mathbf{1}_{\{0 \leq T \leq l\}} \nonumber\\
= &\mathbf{1}_{\{l <T \leq t\}} + \mathbf{1}_{\{0 \leq T \leq t\}} \sum_{k=1}^\xi \hat{N}_{t-T}^l(k), 
\end{align}
where $\hat{N_\tau^l}(k)$ is the number of interior edges of length $\geq l$ in the subtree initiated by $k$th child of the initial particle and which has evolved for a period of time $\tau$ after it has been initiated.

Taking the expectation of the above equation and using the Markov property gives
\begin{align*}
\hat{m}_t^l = &(\mathrm{e}^{- \beta l} - \mathrm{e}^{- \beta t}) + \mathbb{E} [\xi] \int_0^t \hat{m}_{t-s}^l \beta \mathrm{e}^{- \beta s} \mathrm{d}s\\
= &(\mathrm{e}^{- \beta l} - \mathrm{e}^{- \beta t}) + m \beta \mathrm{e}^{-\beta t} \int_0^t \hat{m}_u^l \mathrm{e}^{\beta u}  \mathrm{d}u.
\end{align*}
This translates to the differential equation
\begin{equation*}
\left\{
\begin{array}{rl}
(\hat{m}_t^l)' &= \beta(m-1) \hat{m}_t^l + \beta\mathrm{e}^{-\beta l} \quad ,\ t > l\\
\hat{m}_{l+}^l & = 0
\end{array} \right.
\end{equation*}
(Continuity and differentiability of $\hat{m}_t^l$ follow from the integral equation.) Multiplying the differential equation by the integrating factor $\mathrm{e}^{-\beta(m-1)t}$ gives
\[
(\hat{m}_t^l)' \mathrm{e}^{-\beta(m-1)t} - \beta (m-1) \mathrm{e}^{-\beta(m-1)t} \hat{m}_t^l = \beta \mathrm{e}^{- \beta l} \mathrm{e}^{-\beta(m-1)t}
\]
Integrating this equation over the time interval $[l,t]$ and noting that the left hand side can be written as $\frac{\mathrm{d}}{\mathrm{d}t} \big( \hat{m}_t^l \mathrm{e}^{-\beta (m-1)t} \big)$ we get 
\[
\Big[ \hat{m}_s^l \mathrm{e}^{-\beta (m-1)s} \Big]_l^t = \frac{\mathrm{e}^{-\beta l}}{m-1} \int_l^t \beta (m-1) \mathrm{e}^{-\beta (m-1)s} \mathrm{d}s
\]
and hence using the initial condition $\hat{m}_{l+}^l = 0$ we get
\[
\hat{m}_t^l \mathrm{e}^{-\beta (m-1)t} = \frac{\mathrm{e}^{-\beta l}}{m-1} \Big( \mathrm{e}^{-\beta (m-1)l} - \mathrm{e}^{-\beta (m-1)t} \Big) 
\]
which yields
\[
\hat{m}_t^l = \frac{1}{m-1} \mathrm{e}^{-\beta l} \big( \mathrm{e}^{\beta (m-1)(t- l)} - 1\big)
\]
for $t > l$. Then using the fact that $\tilde{N}^l_t = N^l_t + \hat{N}^l_t$ and recalling the expression for $m_t^l = \mathbb{E}N_t^l$ from \eqref{m_t_l} we get
\[
\tilde{m}_t^l = m_t^l + \hat{m}_t^l = \frac{1}{m-1} \mathrm{e}^{-\beta l} \big( m \mathrm{e}^{\beta (m-1)(t- l)} - 1\big)
\]
for $t > l$.
\end{proof}
One may now see that if $l=l(t)$ is taken to depend on $t$ in any way such that 
\begin{equation}
\label{condition1}
t-l \to \infty 
\end{equation}
as $t \to \infty$ then
\[
\hat{m}_t^l \sim \frac{1}{m-1} \mathrm{e}^{-\beta l} \cdot \mathrm{e}^{\beta(m-1)(t-l)} 
= \frac{1}{m-1} \mathrm{e}^{-\beta (ml - (m-1)t)}
\]
and
\[
\tilde{m}_t^l \sim \frac{m}{m-1} \mathrm{e}^{-\beta l} \cdot \mathrm{e}^{\beta(m-1)(t-l)} = 
\frac{m}{m-1} \mathrm{e}^{-\beta (ml - (m-1)t)}
\]
as $t \to \infty$. Thus, if in addition to \eqref{condition1} we assume that $l=l(t)$ is such that
\begin{equation}
\label{condition2}
ml - (m-1)t \to \infty \quad \text{ as } t \to \infty 
\end{equation}
(for example, if we take $l = \alpha t$, $\alpha \in (\alpha^\ast, 1)$) then $\hat{m}_t^l$ and $\tilde{m}_t^l$ will converge to $0$, whereas if we assume that
\begin{equation}
\label{condition3}
ml - (m-1)t \to - \infty \quad \text{ as } t \to \infty 
\end{equation}
(for example, if we take $l = \alpha t$, $\alpha \in (0, \alpha^\ast$) then $\hat{m}_t^l$ and $\tilde{m}_t^l$ will diverge to $\infty$.

In fact, under conditions \eqref{condition1} and \eqref{condition2} on $l = l(t)$, the decay rates of $\hat{m}_t^l$ and $\tilde{m}_t^l$ will give the decay rates of $\mathbb{P}( \hat{N}_t^l > 0 )$ and $\mathbb{P}( \tilde{N}_t^l > 0 )$ as is stated in the following proposition.
\begin{Proposition}
\label{p_hat_tilde}
Let $l=l(t)$ depend on $t$ in any way such that $t-l \to \infty$ and $ml - (m-1)t \to \infty$ as $t \to \infty$ (e.g., one may have $l = \alpha t + x$ for some $\alpha \in (\alpha^\ast, 1)$ and $x \in \mathbb{R}$). Then 
\begin{equation}
\label{p_hat_limit}
\mathbb{P} \big( \hat{N}_t^l > 0 \big) \sim \hat{m}_t^l \sim \frac{1}{m-1} \mathrm{e}^{-\beta (ml - (m-1)t)} = 
\frac{1}{m-1} \mathrm{e}^{-\beta m(l - \alpha^\ast t)} 
\end{equation}
and
\begin{equation}
\label{p_tilde_limit}
\mathbb{P} \big( \tilde{N}_t^l > 0 \big) \sim \tilde{m}_t^l \sim \frac{m}{m-1} \mathrm{e}^{-\beta (ml - (m-1)t)} = \frac{m}{m-1} \mathrm{e}^{-\beta m(l - \alpha^\ast t)}  
\end{equation}
as $t \to \infty$. 
\end{Proposition}
The proof of Proposition \ref{p_hat_tilde}, due to its length, will be given separately in Section 4. Assuming validity of \eqref{p_hat_limit} and \eqref{p_tilde_limit} we shall now proceed with the proof of Theorem \ref{N_distribution} for interior and all edges. 
\begin{proof}[Proof of Theorem \ref{N_distribution} for interior and all edges]
Let us fix $x \in \mathbb{R}$. For a variable $t \geq 0$ we take
\[
l = l(t) = \alpha^\ast t + x.
\]
We need to show that $\hat{N}_t^l$ and $\tilde{N}_t^l$ converge in distribution to certain mixed Poisson distributions given by \eqref{V_interior} and \eqref{V_all}. We shall do this by proving that the probability generating functions  of $\hat{N}_t^l$ and $\tilde{N}_t^l$ satisfy 
\begin{equation}
\label{pgf_interior}
\mathbb{E} \big[ \theta^{\hat{N}_t^l} \big\vert \text{survival} \big] \to \mathbb{E} \Big[ \exp \big\{ (\theta-1) \frac{\mathrm{e}^{-m \beta x} M_\infty}{m-1} \big\} \Big\vert M_\infty > 0 \Big]
\end{equation}
and
\begin{equation}
\label{pgf_all}
\mathbb{E} \big[ \theta^{\tilde{N}_t^l} \big\vert \text{survival} \big] \to \mathbb{E} \Big[ \exp \big\{ (\theta-1) \frac{m\mathrm{e}^{-m \beta x} M_\infty}{m-1} \big\} \Big\vert M_\infty > 0 \Big]
\end{equation}
as $t \to \infty$ for all $\theta \in (-1, 1)$.

Let us take a small $\delta>0$ ($\delta < \frac{\alpha^\ast}{m}$ will suffice). Then for all $t$ sufficiently large, 
\begin{equation}
\label{N_hat_bounds}
\sum_{k=1}^{N_{\delta t}} \hat{N}_{(1-\delta)t}^l(k) \leq \hat{N}_t^l \leq \sum_{k=1}^{N_{\delta t}} \mathbf{1}_{B_k} + 
\sum_{k=1}^{N_{\delta t}} \hat{N}_{(1-\delta)t}^l(k) 
\end{equation}
and
\begin{equation}
\label{N_tilde_bounds}
\sum_{k=1}^{N_{\delta t}} \tilde{N}_{(1-\delta)t}^l(k) \leq \tilde{N}_t^l \leq \sum_{k=1}^{N_{\delta t}} \mathbf{1}_{B_k} + 
\sum_{k=1}^{N_{\delta t}} \tilde{N}_{(1-\delta)t}^l(k), 
\end{equation}
where for $k \in \{1, 2, \cdots, N_{\delta t}\}$, $\hat{N}_{(1-\delta)t}^l(k)$ and $\tilde{N}_{(1-\delta)t}^l(k)$ are the numbers of interior and all edges of length $\geq l$ in subtrees $\mathcal{T}_{(1-\delta)t}(k)$ initiated by $k$th particle at time $\delta t$ and 
\[
B_k = \big\{ k\text{th particle alive at time } \delta t \text{ does not die in the time interval } [\delta t, l] \big\}.
\]
Since $\mathbb{P}(B_k | \mathcal{F}_{\delta t}) = \mathrm{e}^{-\beta (l - \delta t)}$, $k = 1$, ..., $N_{\delta t}$, it follows that 
\begin{equation}
\label{negligible}
\mathbb{P} \Big( \sum_{k=1}^{N_{\delta t}} \mathbf{1}_{B_k} > 0 \Big\vert \mathcal{F}_{\delta t} \Big) \leq \mathbb{E} \Big[ 
\sum_{k=1}^{N_{\delta t}} \mathbf{1}_{B_k} \big\vert \mathcal{F}_{\delta t} \Big] = \mathrm{e}^{-\beta (l - \delta t)} N_{\delta t} 
= \mathrm{e}^{-\beta(l - m \delta t)} M_{\delta t} \to 0
\end{equation}
$\mathbb{P}$-a.s. as $t \to \infty$.

Let us also consider
\[
\hat{\varphi}^l_{(1-\delta)t} (\theta) = \mathbb{E} \Big[ \theta^{\hat{N}^l_{(1-\delta)t}} \Big] = \sum_{k=0}^\infty \theta^k \mathbb{P} \big(\hat{N}^l_{(1-\delta)t} = k \big)
\]
for $\theta \in (-1, 1)$. We may write it as  
\begin{align*}
\hat{\varphi}^l_{(1-\delta)t} (\theta) = &\mathbb{P}\big(\hat{N}^l_{(1-\delta)t} = 0 \big) + \theta \mathbb{P}\big(\hat{N}^l_{(1-\delta)t} >0 \big) + \sum_{k=2}^\infty (\theta^k - \theta) \mathbb{P}\big(\hat{N}^l_{(1-\delta)t} = k \big)\\
= &1 + (\theta -1) \mathbb{P}\big(\hat{N}^l_{(1-\delta)t} > 0 \big) + \sum_{k=2}^\infty (\theta^k - \theta) \mathbb{P}\big(\hat{N}^l_{(1-\delta)t} = k \big).
\end{align*}
Then recalling from Proposition \ref{p_hat_tilde} that 
\[
\mathbb{P}\big(\hat{N}^l_{(1-\delta)t} > 0 \big) \sim \mathbb{E} \hat{N}^l_{(1-\delta)t} = \hat{m}^l_{(1-\delta)t} \sim 
\frac{1}{m-1} \mathrm{e}^{-\beta (ml - (m-1)(1-\delta)t)} = \frac{\mathrm{e}^{-\beta mx}}{m-1} \mathrm{e}^{-\beta (m-1) \delta t}
\]
as $t \to \infty$ and using the fact that $|\theta^k - \theta| \leq 2$ we see that
\[
\Big\vert \sum_{k=2}^\infty (\theta^k - \theta) \mathbb{P}\big(\hat{N}^l_{(1-\delta)t} = k \big) \Big\vert \leq 2 \mathbb{P}\big(\hat{N}^l_{(1-\delta)t} > 1 \big) \leq 2 \Big( \mathbb{E} \hat{N}^l_{(1-\delta)t} - \mathbb{P}\big(\hat{N}^l_{(1-\delta)t} > 0 \big) \Big) 
= o \big( \mathrm{e}^{-\beta (m-1) \delta t} \big)
\]
as $t \to \infty$ and therefore 
\begin{align}
\label{phi_hat}
\hat{\varphi}^l_{(1-\delta)t} (\theta) = &1 + (\theta - 1) \mathbb{P}\big(\hat{N}^l_{(1-\delta)t} > 0 \big)  + o \big( \mathrm{e}^{-\beta (m-1) \delta t} \big)\nonumber\\
= &1 + (\theta - 1) \frac{\mathrm{e}^{-\beta mx}}{m-1} \mathrm{e}^{-\beta (m-1) \delta t}  + o \big( \mathrm{e}^{-\beta (m-1) \delta t} \big)
\end{align}
Then using \eqref{N_hat_bounds}, \eqref{negligible}, the facts that $\mathbb{E} \big[ \theta^{\sum_{k=1}^{N_{\delta t}} \hat{N}_{(1-\delta)t}^l(k)} \big\vert \mathcal{F}_{\delta t} \big] = \big( \hat{\varphi}^l_{(1-\delta)t} (\theta) \big)^{N_{\delta t}}$,

 $\{N_{\delta t} > 0\} \searrow \{$survival$\}$ and $\mathbb{P}(M_\infty > 0 | $survival$) = 1$ and bounded convergence we see that 
\begin{align*}
\mathbb{E} \Big[ \theta^{\hat{N}^l_t} \big\vert \text{survival} \Big] 
= &\mathbb{E} \Big[ \theta^{\hat{N}^l_t} \big\vert N_{\delta t} > 0 \Big] + o(1)\\
= &\mathbb{E} \Big[ \theta^{\sum_{k=1}^{N_{\delta t}} \hat{N}_{(1-\delta)t}^l(k)} \big\vert N_{\delta t} > 0 \Big] + o(1)\\ 
= &\mathbb{E} \Big[ \Big( \hat{\varphi}^l_{(1-\delta)t} (\theta) \Big)^{N_{\delta t}} \big\vert N_{\delta t} > 0 \Big] + o(1)\\
= &\mathbb{E} \Big[ \Big( \hat{\varphi}^l_{(1-\delta)t} (\theta) \Big)^{N_{\delta t}} \big\vert \text{survival} \Big] + o(1)\\
= &\mathbb{E} \Big[ \Big( 1 + (\theta - 1) \frac{\mathrm{e}^{-\beta mx}}{m-1} \mathrm{e}^{-\beta (m-1) \delta t}  + o \big( \mathrm{e}^{-\beta (m-1) \delta t} \big) \Big)^{\mathrm{e}^{\beta (m-1)\delta t} M_{\delta t}} \big\vert M_\infty > 0 \Big] + o(1)\\
\to &\mathbb{E} \Big[ \exp \Big\{ (\theta-1) \frac{\mathrm{e}^{-m \beta x}}{m-1} M_\infty \Big\} \big\vert M_\infty>0 \Big] 
\end{align*}
as $t \to \infty$, which proves \eqref{pgf_interior}. Convergence in \eqref{pgf_all} is proved in exactly the same way.
\end{proof}
\section{Proof of Proposition \ref{p_hat_tilde}}
Recall that in the previous section we have calculated $\hat{m}_t^l = \mathbb{E} \hat{N}_t^l$ and $\tilde{m}_t^l = \mathbb{E} \tilde{N}_t^l$, which was relatively easy. Let us now consider $\hat{v}_t^l = \mathbb{E} [(\hat{N}_t^l )^2]$ and $\tilde{v}_t^l = \mathbb{E} [( \tilde{N}_t^l )^2 ]$. Finding the exact expressions for both of these quantities is quite possible (as will be apparent from the calculations below) but they are cumbersome. Instead, we will find simple upper bounds on $\hat{v}_t^l$ and $\tilde{v}_t^l$. These upper bounds will be used to tell us that if $v<\infty$ then under conditions \eqref{condition1} and \eqref{condition2} on $l$, $\hat{m}_t^l \sim \hat{v}_t^l$ and $\tilde{m}_t^l \sim \tilde{v}_t^l$ as $t \to \infty$. This together with a suitable truncation of the branching process $\mathcal{T}_t$, $t \geq 0$ will enable us to prove Proposition \ref{p_hat_tilde}.  
\begin{Proposition}
\label{v_hat_tilde}
Let us suppose that $v = \mathbb{E} [\xi^2] = \sum_{k \geq 0} k^2 p_k < \infty$. If $t \geq l$ then
\begin{equation}
\label{v_hat_upper}
\hat{v}_t^l \leq \frac{1}{m-1} \mathrm{e}^{-\beta(ml - (m-1)t)} \Big[ 1+ 2 \beta m (t-l) \mathrm{e}^{- \beta m l} + 
\frac{v-m}{(m-1)^2} \mathrm{e}^{-\beta(ml - (m-1)t)} \Big]
\end{equation}
and 
\begin{equation}
\label{v_tilde_upper}
\tilde{v}_t^l \leq \frac{m}{m-1} \mathrm{e}^{-\beta(ml - (m-1)t)} \Big[ 1+ 2 \beta (t-l) \mathrm{e}^{- \beta m l} + 
\frac{(v-m)(m^2 - 2m + 2)}{m(m-1)^2} \mathrm{e}^{-\beta(ml - (m-1)t)} \Big].
\end{equation}
\end{Proposition} 
\begin{proof}
We follow the same steps as in the proof of Proposition \ref{m_hat_tilde}. Let us fix $l$ and let $t > l$. From \eqref{N} we get that
\[
\big( \hat{N}_t^l \big)^2 = \mathbf{1}_{\{l <T \leq t\}} + 2 \cdot \mathbf{1}_{\{l \leq T \leq t\}} \sum_{k=1}^\xi \hat{N}_{t-T}^l(k) 
+ \mathbf{1}_{\{0 \leq T \leq t\}} \Big( \sum_{k=1}^\xi \hat{N}_{t-T}^l(k) \Big)^2
\]
Taking the expectation and using the Markov property we get the integral equation for $\hat{v}_t^l$:
\begin{align*}
\hat{v}_t^l = &(\mathrm{e}^{- \beta l} - \mathrm{e}^{- \beta t}) + 2 \mathbb{E}(\xi) \int_l^t \hat{m}_{t-s}^l \beta \mathrm{e}^{- \beta s} \mathrm{d}s + \int_0^t \big( \mathbb{E}(\xi) \cdot \hat{v}_{t-s}^l + \mathbb{E}(\xi^2 - \xi) \cdot (\hat{m}_{t-s}^l)^2 \big) \beta \mathrm{e}^{- \beta s} \mathrm{d}s\\
= &(\mathrm{e}^{- \beta l} - \mathrm{e}^{- \beta t}) + 2m \beta \mathrm{e}^{-\beta t} \int_0^t \hat{m}_u^l \mathrm{e}^{\beta u} \mathrm{d}u + \beta \mathrm{e}^{- \beta t} \int_0^t \big( m \hat{v}_u^l +(v-m)(\hat{m}_u^l)^2\big) \mathrm{e}^{\beta u} \mathrm{d}u
\end{align*}
Differentiating it with respect to $t$ yields the differential equation 
\begin{equation*}
\left\{
\begin{array}{rl}
(\hat{v}_t^l)' &= \beta(m-1) \hat{v}_t^l + \beta\mathrm{e}^{-\beta l} + 2 \beta m \mathrm{e}^{-\beta l} \hat{m}_{t-l}^l + \beta (v-m) (\hat{m}_t^l)^2 \quad ,\ t > l\\
\hat{v}_{l+}^l & = 0
\end{array} \right.
\end{equation*}
Multiplying by the integrating factor $\mathrm{e}^{- \beta (m-1) t}$ gives
\begin{align*}
(\hat{v}_t^l)' \mathrm{e}^{- \beta (m-1) t} - \beta (m-1) \mathrm{e}^{- \beta (m-1) t} \hat{v}_t^l = 
&\beta \mathrm{e}^{- \beta l} \mathrm{e}^{- \beta (m-1) t} + 2 \beta m \mathrm{e}^{- \beta l} \mathrm{e}^{- \beta (m-1) t} \hat{m}_{t-l}^l\\
&+ \beta (v-m) \mathrm{e}^{- \beta (m-1) t} (\hat{m}_t^l)^2 
\end{align*}
Integrating this equation between $l$ and $t$ and then multiplying by $\mathrm{e}^{\beta (m-1) t}$ gives 
\begin{align}
\label{v}
\hat{v}_t^l = &\mathrm{e}^{\beta (m-1) t} \int_l^t \mathrm{e}^{- \beta (m-1) s} \Big( \beta \mathrm{e}^{- \beta l} + 
2 \beta m \mathrm{e}^{- \beta l} \hat{m}_{s-l}^l + \beta (v-m) (\hat{m}_s^l)^2 \Big) \mathrm{d}s \nonumber\\
= &\mathrm{e}^{\beta (m-1) t} \Big[ \beta \mathrm{e}^{- \beta m l} \int_l^t \mathrm{e}^{- \beta (m-1) (s-l)} \mathrm{d}s \nonumber\\ 
&\qquad\qquad + \frac{2 \beta m}{m-1} \mathrm{e}^{- 2\beta ml} \int_{2l}^t \mathrm{e}^{- \beta (m-1)(s-2l)} \big( \mathrm{e}^{\beta (m-1)(s-2l)} - 1 \big) \mathrm{d}s \mathbf{1}_{\{t \geq 2l\}} \nonumber\\
&\qquad\qquad + \frac{\beta(v-m)}{(m-1)^2} \mathrm{e}^{- \beta (m+1)l} \int_l^t \mathrm{e}^{-\beta (m-1)(s-l)} \big( \mathrm{e}^{\beta(m-1)(s-l)} - 1\big)^2 \mathrm{d}s \Big]\nonumber\\
= &\frac{1}{m-1} \mathrm{e}^{-\beta (ml - (m-1)t)} \Big[ \beta(m-1) \int_l^t \mathrm{e}^{-\beta (m-1)(s-l)} \mathrm{d}s\nonumber\\
&\qquad\qquad\qquad\qquad\qquad + 2 \beta m \mathrm{e}^{- \beta m l} \int_{2l}^t 1 - \mathrm{e}^{-\beta (m-1)(s-2l)} \mathrm{d}s \mathbf{1}_{\{t \geq 2l\}}\nonumber\\
&\qquad\qquad\qquad\qquad\qquad + \frac{\beta (v-m)}{m-1} \mathrm{e}^{- \beta l} \int_l^t \mathrm{e}^{-\beta(m-1)(s-l)} \big( \mathrm{e}^{\beta(m-1)(s-l)} - 1 \big)^2 \mathrm{d}s \Big].
\end{align}
Let us now give an upper bound on each of the three terms inside the square brackets, which we refer to as (I), (II) and (III):
\[
(I) = \int_l^t \mathrm{e}^{-\beta(m-1)(s-l)} \beta (m-1) \mathrm{d}s \leq \int_0^\infty \mathrm{e}^{-\beta(m-1) \tau} \beta (m-1) \mathrm{d}\tau = 1,
\]
\[
(II) = 2 \beta m \mathrm{e}^{-\beta m l} \int_{2l}^t 1 - \mathrm{e}^{-\beta(m-1)(s-2l)} \mathrm{d}s \mathbf{1}_{\{t \geq 2l\}} \leq 
2 \beta m \mathrm{e}^{-\beta m l}(t - 2l) \mathbf{1}_{\{t \geq 2l\}} \leq 2 \beta m \mathrm{e}^{-\beta m l} (t-l)
\]
and
\begin{align*}
(III) \leq &\frac{\beta (v-m)}{m-1} \mathrm{e}^{- \beta l} \int_l^t \mathrm{e}^{-\beta (m-1)(s-l)} \big( \mathrm{e}^{\beta(m-1)(s-l)} - 1\big)^2 \mathrm{d}s\\
= &\frac{v-m}{(m-1)^2} \mathrm{e}^{-\beta l} \int_0^{t-l} \mathrm{e}^{-\beta(m-1) \tau} \big( \mathrm{e}^{\beta(m-1)\tau} - 1 \big)^2 \beta (m-1) \mathrm{d} \tau\\
\leq &\frac{v-m}{(m-1)^2} \mathrm{e}^{-\beta l} \int_0^{t-l} \mathrm{e}^{\beta (m-1) \tau} \beta (m-1) \mathrm{d} \tau\\
\leq &\frac{v-m}{(m-1)^2} \mathrm{e}^{-\beta l} \mathrm{e}^{\beta(m-1)(t-l)}\\
= &\frac{v-m}{(m-1)^2} \mathrm{e}^{-\beta(ml - (m-1)t)},
\end{align*}
which establishes \eqref{v_hat_upper}. Inequality \eqref{v_tilde_upper} is then easily derived from \eqref{v_hat_upper} and \eqref{v_t_l} using the fact that $\tilde{v}_t^l \leq \hat{v}_t^l + v_t^l$.
\end{proof}
Let us now apply upper bounds \eqref{v_hat_upper} and \eqref{v_tilde_upper} on $\hat{v}_t^l$ and $\tilde{v}_t^l$ to prove Proposition \ref{p_hat_tilde}. Namely, that if $l$ depends on $t$ in any such way that $t-l \to \infty$ and $ml - (m-1)t \to \infty$ as $t \to \infty$ (conditions \eqref{condition1} and \eqref{condition2}) then 
\[
\mathbb{P} \big( \hat{N}_t^l > 0 \big) \sim \hat{m}_t^l \sim \frac{1}{m-1} \mathrm{e}^{-\beta (ml - (m-1)t)} 
\]
and
\[
\mathbb{P} \big( \tilde{N}_t^l > 0 \big) \sim \tilde{m}_t^l \sim \frac{m}{m-1} \mathrm{e}^{-\beta (ml - (m-1)t)}.
\]
\begin{proof}[Proof of Proposition \ref{p_hat_tilde}]
Trivially, $\mathbb{P} \big( \hat{N}_t^l > 0 \big) \leq \mathbb{E} \hat{N}_t^l = \hat{m}_t^l$ and $\mathbb{P} \big( \tilde{N}_t^l > 0 \big) \leq \mathbb{E} \tilde{N}_t^l = \tilde{m}_t^l$. Therefore, 
\begin{equation}
\label{p_hat_upper}
\limsup_{t \to \infty} \mathrm{e}^{\beta (ml - (m-1)t)} \mathbb{P} \big( \hat{N}_t^l > 0 \big) \leq \frac{1}{m-1}
\end{equation}
and 
\begin{equation}
\label{p_tilde_upper}
\limsup_{t \to \infty} \mathrm{e}^{\beta (ml - (m-1)t)} \mathbb{P} \big( \tilde{N}_t^l > 0 \big) \leq \frac{m}{m-1}.
\end{equation}
To establish corresponding lower bounds let us consider two cases: $v<\infty$ and $v=\infty$.

\underline{Case (I): $v < \infty$}:
Note that under conditions \eqref{condition1} and \eqref{condition2} inequalities \eqref{v_hat_upper} and \eqref{v_tilde_upper} say that 
\[
\hat{v}_t^l \leq \frac{1}{m-1} \mathrm{e}^{-\beta(ml - (m-1)t)} \Big[ 1+ o(1) \Big] \sim \hat{m}^l_t
\]
and 
\[
\tilde{v}_t^l \leq \frac{m}{m-1} \mathrm{e}^{-\beta(ml - (m-1)t)} \Big[ 1+ o(1) \Big] \sim \tilde{m}^l_t
\]
as $t \to \infty$. Then, since also $\hat{v}_t^l \geq \hat{m}_t^l$ and $\tilde{v}_t^l \geq \tilde{m}_t^l$, it follows that $\hat{v}_t^l \sim \hat{m}_t^l$ and $\tilde{v}_t^l \sim \tilde{m}_t^l$. Hence by Paley-Zygmund inequality
\[
\mathbb{P} \big( \hat{N}_t^l > 0 \big) \geq \frac{\big( \mathbb{E} \hat{N}_t^l \big)^2}{\mathbb{E} \big[ \big( \hat{N}_t^l \big)^2\big]} = 
\frac{(\hat{m}_t^l)^2}{\hat{v}_t^l} \sim \hat{m}^l_t
\]
as $t \to \infty$ and likewise for $\mathbb{P} \big( \tilde{N}_t^l > 0 \big)$.

\underline{Case (II): $v = \infty$}:

The issue here is that now $\hat{v}_t^l = \tilde{v}^l_t = \infty$ and so Paley-Zygmund inequality is of no help to us. To overcome this problem we consider the truncated process $\mathcal{T}_t^K$, $t \geq 0$ derived from $\mathcal{T}_t$, $t \geq 0$ by having all the offspring killed at each birth event which produces more than $K$ particles (where $K \geq 1$ is some given number). Then every edge in $\mathcal{T}_t^K$ is an edge in $\mathcal{T}_t$ and hence $\mathbb{P} \big( \hat{N}_t^l > 0\big) \geq \mathbb{P} \big( \hat{N}_t^{K,l} > 0\big)$ and $\mathbb{P} \big( \tilde{N}_t^l > 0\big) \geq \mathbb{P} \big( \tilde{N}_t^{K,l} > 0\big)$, where $\hat{N}_t^{K,l}$ and $\tilde{N}_t^{K,l}$ are numbers of interior and all (interior and pendant) edges of length $\geq l$ in $\mathcal{T}_t^K$.

Moreover, the truncated process $\mathcal{T}_t^K$, $t \geq 0$ is a Galton-Watson process with branching rate $\beta$ and offspring distribution $p^K_k$, $k \geq 0$, where 
\[
p_0^K = \mathbb{P}(\xi = 0) + \mathbb{P}(\xi > K) = p_0 + p_{K+1} + p_{K+2} + \cdots,
\]
\[
p_1^K = \mathbb{P}(\xi = 1) = p_1, \ p_2^K = \mathbb{P}(\xi = 2) = p_2, \ \cdots , \ p_K^K = \mathbb{P}(\xi = K) = p_K
\]
and
\[
p_{K+1}^K = p_{K+2}^K = \cdots = 0.
\]
If we let $m^K = \sum_{k \geq 0} k p_k^K = \sum_{k=0}^K k p_k$ and $v^K = \sum_{k\geq 0} k^2 p_k^K = \sum_{k=0}^K k^2 p_k$ denote the first and second moments of the offspring distribution in the truncated Galton-Watson process then we see that $v^K \leq K^2$ and 
\[
m - m^K = \mathbb{E} \big[ \xi \mathbf{1}_{\{\xi > K\}} \big] \leq \mathbb{E} \Big[ \xi \frac{\log^+ \xi}{\log K} \mathbf{1}_{\{\xi > K\}} \Big] \leq \frac{F(K)}{\log K},
\]
where $F(K) = \mathbb{E} \big[ \xi \log^+ \xi \mathbf{1}_{\{\xi > K\}} \big]$ is a function such that $F(K) \to 0$ as $K \to \infty$. Hence
\begin{align}
\label{m_K}
m - \frac{F(K)}{\log K} \leq m^K \leq m.
\end{align}
We then have that 
\begin{equation}
 \label{paley_zygmund}
\mathbb{P} \big( \hat{N}_t^l > 0\big) \geq \mathbb{P} \big( \hat{N}_t^{K,l} > 0\big) \geq \frac{\big( \mathbb{E} \hat{N}_t^{K,l} \big)^2}{\mathbb{E} \big[ \big( \hat{N}_t^{K,l} \big)^2\big]}
\end{equation}
for any $K \geq 1$. Let us now take $K = K(t) = \lfloor \mathrm{e}^{c(ml - (m-1)t)} \rfloor$ for any $c \in (0, \frac{\beta}{2})$.
Then from \eqref{m_hat_tilde} we have that 
\[
\mathbb{E} \hat{N}_t^{K,l} = \frac{1}{m^K-1} \mathrm{e}^{-\beta l} \big( \mathrm{e}^{\beta (m^K-1)(t-l)} -1 \big).
\]
From \eqref{m_K} we know that 
\[
m(t-l) - \frac{F(K)(t-l)}{ct - cm(t-l)} \leq m^K(t-l) \leq m(t-l)
\]
from which it follows that $m^K(t-l) = m(t-l) + o(1)$ and hence 
\[
\mathrm{e}^{\beta (m^K-1)(t-l)} \sim \mathrm{e}^{\beta (m-1)(t-l)}
\]
as $t \to \infty$. Thus 
\begin{equation}
\label{m_hat_K}
\mathbb{E} \hat{N}_t^{K,l} \sim \frac{1}{m-1} \mathrm{e}^{-\beta l} \mathrm{e}^{\beta(m-1)(t-l)} = \frac{1}{m-1} \mathrm{e}^{-\beta (ml - (m-1)t)}
\end{equation}
as $t \to \infty$. Also, from \eqref{v_hat_upper} we have that 
\begin{align}
\label{v_hat_K_upper}
\mathbb{E} \hat{N}_t^{K,l} \leq \mathbb{E} \Big[ \big( \hat{N}_t^{K,l} \big)^2 \Big] \leq \frac{1}{m^K-1} \mathrm{e}^{-\beta(m^Kl - (m^K-1)t)} \Big[ 1+ &2 \beta m^K (t-l) \mathrm{e}^{- \beta m^K l}\nonumber\\ 
+ &\frac{v^K-m^K}{(m^K-1)^2} \mathrm{e}^{-\beta(m^Kl - (m^K-1)t)} \Big].
\end{align}
For the same reason as above, 
\[
\mathrm{e}^{-\beta(m^Kl - (m^K-1)t)} \sim \mathrm{e}^{-\beta(ml - (m-1)t)}
\]
as $t \to \infty$. We see that the second term inside the square brackets of \eqref{v_hat_K_upper} converges to $0$ as $t \to \infty$. So does the third term since
\begin{align*}
v^K \mathrm{e}^{-\beta(m^Kl - (m^K-1)t)} \leq K^2 \mathrm{e}^{-\beta(m^Kl - (m^K-1)t)} \leq &\mathrm{e}^{2c(ml - (m-1)t)} \mathrm{e}^{-\beta(m^Kl - (m^K-1)t)}\\
\leq &\mathrm{e}^{(2c -\beta)(ml - (m-1)t)} \to 0
\end{align*}
as $t \to \infty$. It follows that
\begin{align}
\label{v_hat_K}
\mathbb{E} \Big[ \big( \hat{N}_t^{K,l} \big)^2 \Big] \sim \frac{1}{m-1} \mathrm{e}^{-\beta(ml - (m-1)t)}.
\end{align}
Combining \eqref{paley_zygmund}, \eqref{m_hat_K} and \eqref{v_hat_K} yields
\begin{equation}
\label{p_hat_lower}
\liminf_{t \to \infty} \mathrm{e}^{\beta (ml - (m-1)t)} \mathbb{P} \big( \hat{N}_t^l > 0 \big) \geq \frac{1}{m-1}.
\end{equation}
Similarly it can be shown that 
\begin{equation}
\label{p_tilde_lower}
\liminf_{t \to \infty} \mathrm{e}^{\beta (ml - (m-1)t)} \mathbb{P} \big( \tilde{N}_t^l > 0 \big) \geq \frac{m}{m-1}.
\end{equation}
Inequalities \eqref{p_hat_upper}, \eqref{p_hat_lower}, \eqref{p_tilde_upper} and \eqref{p_tilde_lower} establish asymptotic relations \eqref{p_hat_limit} and \eqref{p_tilde_limit}.
\end{proof}
\section{Proof of Theorem \ref{a_s}}
In this last section we prove that $\frac{L_t^{(k)}}{t}$, $\frac{\hat{L}_t^{(k)}}{t}$ and $\frac{\tilde{L}_t^{(k)}}{t}$ for any choice of $k = 1, 2, 3, $... all converge to $\alpha^\ast$  as $t \to \infty$ a.s. on survival. 

We do this by first establishing convergence along integers and then expanding it to convergence along real numbers.
\begin{proof}[Proof of convergence in Theorem \ref{a_s} along integers]
We choose $k = 1, 2, 3, $... and observe that for any $\alpha \in (\alpha^\ast, 1)$ and $n \in \mathbb{N}$
\[
\mathbb{P} \big( \frac{L^{(k)}_n}{n} \geq \alpha \big) = \mathbb{P} \big( L^{(k)}_n \geq \alpha n \big) = \mathbb{P} \big( N_n^{\alpha n} \geq k \big) \leq \frac{1}{k} \mathbb{E} N_n^{\alpha n} = \frac{1}{k} \mathrm{e}^{- \beta m (\alpha - \alpha^\ast)n},
\]
which is summable over $n$ and hence by Borel-Cantelli Lemma $\mathbb{P} \big( \frac{L^{(k)}_n}{n} \geq \alpha \text{ i.o.}\big) = 0$ and also $\mathbb{P} \big( \frac{L^{(k)}_n}{n} \geq \alpha \text{ i.o.}| $survival$\big) = 0$. Therefore for any $\alpha \in (\alpha^\ast, 1)$
\[
\limsup_{n \to \infty} \frac{L^{(k)}_n}{n} \leq \alpha
\]
a. s. on survival and hence 
\begin{equation}
\label{limsup}
\limsup_{n \to \infty} \frac{L^{(k)}_n}{n} \leq \alpha^\ast
\end{equation}
a. s. on survival. The same argument applies to $\hat{L}^{(k)}_n$ and $\tilde{L}^{(k)}_n$.

On the other hand, we may take any $\alpha \in (0, \alpha^\ast)$ and let $\delta = 1 - \frac{\alpha}{\alpha^\ast}$ (so that $\alpha = \alpha^\ast (1- \delta)$). Then 
\[
\big\{ N_n^{\alpha n} < k \big\} \subseteq \bigcap_{j=1}^{N_{\delta n}} \big\{ N_{(1- \delta)n}^{\alpha n}(j) < k \big\},
\]
where $N_{(1- \delta)n}^{\alpha n}(j) = N_{(1- \delta)n}^{\alpha^\ast (1-\delta) n}(j)$ are the numbers of pendant edges of length $\geq \alpha n$ in subtrees $\mathcal{T}_{(1-\delta)n}(j)$ initiated by particles $j = 1$, $2$, ..., $N_{\delta n}$ at time $\delta n$ and evolved up to time $n$. Then we check that 
\begin{align*}
\mathbb{P} \big( \frac{L_n^{(k)}}{n} < \alpha, \ N_{\delta n} > n\big) = &\mathbb{P} \big( N_n^{\alpha n} < k, \ N_{\delta n} > n\big)\\
\leq &\mathbb{P} \Big( \bigcap_{j=1}^{N_{\delta n}} \big\{ N_{(1- \delta)n}^{\alpha n}(j) < k \big\} , \ N_{\delta n} > n\Big)\\
= &\mathbb{E} \Big[ \Big( \mathbb{P} \big( N_{(1- \delta)n}^{\alpha n} < k \big) \Big)^{N_{\delta n}} \mathbf{1}_{\{N_{\delta n} > n\}} \Big]\\
\leq &\mathbb{P} \big( N_{(1- \delta)n}^{\alpha^\ast (1-\delta) n} < k \big)^n \mathbb{P}\big(N_{\delta n} > n \big).
\end{align*}
Since $\mathbb{P} \big( N_{(1- \delta)n}^{\alpha^\ast (1-\delta) n} < k \big) \to \mathbb{P}(V_0 < k) < 1$ as $n \to \infty$ it follows that $\mathbb{P} ( \frac{L_n^{(k)}}{n} < \alpha, \ N_{\delta n} > n)$ is summable over $n$. Therefore 
\[
\mathbb{P} \big( \frac{L_n^{(k)}}{n} < \alpha, \ N_{\delta n} > n \text{ i.o.}\big) = 0
\]
and so also 
\[
\mathbb{P} \big( \frac{L_n^{(k)}}{n} < \alpha, \ N_{\delta n} > n \text{ i.o.} \big\vert \text{survival} \big) = 0.
\]
Since 
\[
\mathbb{P} \big( N_{\delta n} > n \text{ ev.} \big\vert \text{survival} \big) = 1.
\]
it must be that 
\[
\mathbb{P} \big( \frac{L_n^{(k)}}{n} < \alpha \text{ i.o.} \big\vert \text{survival} \big) = 0.
\]
and thus 
\[
\liminf_{n \to \infty} \frac{L_n^{(k)}}{n} \geq \alpha
\]
a.s. on survival for all $\alpha \in (\alpha^\ast, 1)$. Therefore 
\begin{equation}
\label{liminf}
\liminf_{n \to \infty} \frac{L_n^{(k)}}{n} \geq \alpha^\ast
\end{equation}
a.s. on survival and the same argument applies to $\hat{L}_n^{(k)}$ and $\tilde{L}_n^{(k)}$. Putting together \eqref{limsup} and \eqref{liminf} we see that 
\[
\lim_{n \to \infty} \frac{L_n^{(k)}}{n} = \lim_{n \to \infty} \frac{\hat{L}_n^{(k)}}{n} = \lim_{n \to \infty} \frac{\tilde{L}_n^{(k)}}{n} = \alpha^\ast
\] 
a.s. on survival.
\end{proof}

\begin{proof}[Proof of convergence in Theorem \ref{a_s} along integers]
The key observations are that for any times $s \leq t$ it is true that
\[
L_t^{(k)} \leq L_s^{(k)} + (t-s)
\]
(since over the time interval $[s,t]$ the $k$th longest pendant edge cannot increase by more than $t-s$),
\[
\hat{L}_s^{(k)} \leq \hat{L}_t^{(k)}
\]
and
\[
\tilde{L}_s^{(k)} \leq \tilde{L}_t^{(k)}
\]
(since $k$th longest edge and $k$th longest interior edge can not get shorter over time). Thus for any $t \in [0, \infty)$ it is true that 
\[
\frac{L_{\lceil t \rceil}^{(k)}}{\lceil t \rceil} \cdot \frac{\lceil t \rceil}{t} - \frac{1}{t} = \frac{L_{\lceil t \rceil}^{(k)} - 1}{t} \leq \frac{L_t^{(k)}}{t} \leq \frac{L_{\lfloor t \rfloor}^{(k)} + 1}{t} = \frac{L_{\lfloor t \rfloor}^{(k)}}{\lfloor t \rfloor} \cdot \frac{\lfloor t \rfloor}{t} + \frac{1}{t},
\]
\[
\frac{\hat{L}_{\lfloor t \rfloor}^{(k)}}{\lfloor t \rfloor} \cdot \frac{\lfloor t \rfloor}{t} = \frac{\hat{L}_{\lfloor t \rfloor}^{(k)}}{t} \leq \frac{\hat{L}_t^{(k)}}{t} \leq \frac{\hat{L}_{\lceil t \rceil}^{(k)}}{t} = \frac{\hat{L}_{\lceil t \rceil}^{(k)}}{\lceil t \rceil} \cdot \frac{\lceil t \rceil}{t}
\]
and
\[
\frac{\tilde{L}_{\lfloor t \rfloor}^{(k)}}{\lfloor t \rfloor} \cdot \frac{\lfloor t \rfloor}{t} = \frac{\tilde{L}_{\lfloor t \rfloor}^{(k)}}{t} \leq \frac{\tilde{L}_t^{(k)}}{t} \leq \frac{\tilde{L}_{\lceil t \rceil}^{(k)}}{t} = \frac{\tilde{L}_{\lceil t \rceil}^{(k)}}{\lceil t \rceil} \cdot \frac{\lceil t \rceil}{t}.
\]
Taking $t \to \infty$ yields the results.
\end{proof}

\end{document}